\documentclass[11pt]{amsart}
\usepackage[cm]{fullpage}
\usepackage{amssymb,amsmath,amsthm,amscd,mathrsfs,graphicx}
\usepackage[cmtip,all]{xy}
\usepackage{pb-diagram}

\numberwithin{equation}{section}

\newtheorem{theorem}{Theorem}[section]
\newtheorem{proposition}{Proposition}[section]
\newtheorem{lemma}{Lemma}[section]
\newtheorem{corollary}{Corollary}[section]

\newcommand{\al}{\alpha}

\newcommand{\ov}[1]{\overline{#1}}

\newcommand{\ve}{\varepsilon}

\theoremstyle{definition}

\theoremstyle{remark}

\begin{document}
\bibliographystyle{amsplain}

\title{Weak Harnack inequalities for eigenvalues \\ and constant rank theorems}

\begin{abstract}  We consider convex solutions of nonlinear elliptic equations which satisfy the  structure condition of Bian-Guan.  
We prove a weak Harnack inequality for the eigenvalues of the Hessian of these solutions.  This can be viewed as a quantitative version of the constant rank theorem. 

\end{abstract}

\author[G. Sz\'ekelyhidi]{G\'abor Sz\'ekelyhidi}
\address{Department of Mathematics, University of Notre Dame, 255 Hurley, Notre Dame, IN 46556}
\author[B. Weinkove]{Ben Weinkove}
\address{Department of Mathematics, Northwestern University, 2033 Sheridan Road, Evanston, IL 60208}
\thanks{Supported in part by National Science Foundation grants DMS-1906216  and DMS-2005311. Part of this work was carried out while the second-named author was visiting the Department of Mathematical Sciences at the University of Memphis and he thanks them for their kind hospitality.}

\maketitle

\section{Introduction}

Constant rank theorems in PDE have a long history, starting with work of Caffarelli-Friedman \cite{CF}, Yau (see \cite{SWYY}) and then developed further by Korevaar-Lewis \cite{KL}, Caffarelli-Guan-Ma \cite{CGM}, Bian-Guan \cite{BG, BG2} and others \cite{BGMX, GLZ, GP, HMW, L}.  These results assert that a convex solution $u$ of a certain class of elliptic or parabolic equations has Hessian $D^2u$ of constant rank.  

Constant rank theorems, also known as the ``microscopic convexity principle'', have been used to establish ``macroscopic'' convexity properties of solutions to PDEs on convex domains, now a vast area of research (see \cite{ALL, Bo, BL, CF, CSp, CMY, CW, CW2, CH, CMS,  G, GM, Ish,  Ka, Ke, K,   MS, RR, W, We} and the references therein).

One method for establishing a constant rank theorem is to compute with an expression involving the elementary symmetric polynomials $\sigma_k$ of the eigenvalues $$\lambda_1 \le \cdots \le \lambda_n$$ of the Hessian $D^2u$.  Bian-Guan \cite{BG} considered solutions of nonlinear elliptic equations
$$F(D^2 u, Du, u, x)=0,$$
under a convexity condition for $F$ (see (\ref{ellipticitycon}) below), and proved a constant rank theorem using a differential inequality for the quantity $\sigma_{\ell+1} +\frac{\sigma_{\ell+2}}{\sigma_{\ell+1}}$.  The authors \cite{SW} gave a new proof of the Bian-Guan result using the simple linear expression
\begin{equation} \label{simple}
\lambda_{\ell} + 2\lambda_{\ell-1} + \cdots + \ell \lambda_1,
\end{equation}
and a method of induction.  This approach exploited the concavity of the sums of the lowest eigenvalues.

In this paper we will also assume the Bian-Guan structure conditions.   Building on  the method of \cite{SW} and making use again of the expression (\ref{simple}) we directly prove a weak Harnack inequality for each of the eigenvalues $\lambda_i$.   This states that the $L^q$ norm for some $q>0$ is bounded above by the infimum.  We view this Harnack inequality as a \emph{quantitative} version of the constant rank theorem of Bian-Guan, which follows as an immediate consequence.

Another difference between the current paper and \cite{SW} is that we compute differential inequalities directly, holding almost everywhere, for sums of the eigenvalues $\lambda_i$. 
 In particular we avoid here the device of approximation by polynomials.

We now state our results precisely.  Let $B=B_1(0)$ be the unit ball in $\mathbb{R}^n$ and let $F$ be a real-valued function
$$F=F(A, p, u, x) \in C^2(\textrm{Sym}_n(\mathbb{R}) \times \mathbb{R}^n \times \mathbb{R} \times B),$$
where $\textrm{Sym}_n(\mathbb{R})$ is the vector space of  symmetric $n\times n$ matrices with real entries.  We assume that $F$ satisfies the condition of Bian-Guan \cite{BG} that for each $p \in \mathbb{R}^n$,
\begin{equation} \label{convexcondition}
(A, u,x) \in \textrm{Sym}_n^+(\mathbb{R}) \times \mathbb{R} \times B \mapsto F(A^{-1}, p, u, x) \quad \textrm{is locally convex},
\end{equation}
where $\textrm{Sym}_n^+(\mathbb{R})$ is the subset of $\textrm{Sym}_n(\mathbb{R})$ that are strictly positive definite.  Suppose that $u \in C^3(B)$ is a convex solution of 
\begin{equation} \label{maineqn}
F(D^2u, Du, u, x)=0,
\end{equation}
subject to the ellipticity condition that for all $\xi \in \mathbb{R}^n$,
\begin{equation} \label{ellipticitycon}
\Lambda^{-1} |\xi|^2 \le F^{ij} (D^2u, Du, u,x)\xi^i \xi^j \le \Lambda |\xi|^2, \quad \textrm{on } B,
\end{equation}
for a positive constant $\Lambda>0$, where $F^{ij}$ is the derivative of $F$ with respect to the $(i,j)$th entry $A_{ij}$ of $A$.  Our main result is as follows.

\begin{theorem} \label{mainthm}
Let $u$ be as above and let $0 \le \lambda_1 \le \cdots \le \lambda_n$ be the eigenvalues of $D^2u$.  Then there exist positive constants $C_0, q$  depending only on $n$, $\Lambda$, $\| u \|_{C^3(B)}$ and $\| F \|_{C^2}$ such that for each $\ell =1, \ldots, n$,
$$\| \lambda_{\ell} \|_{L^q(B_{1/2})}  \le C_0 \inf_{B_{1/2}} \lambda_{\ell},$$
where $B_{1/2}=B_{1/2}(0)$ is the ball in $\mathbb{R}^n$ centered at $0$ of radius $1/2$. 
\end{theorem}

This implies in particular the constant rank theorem of Bian-Guan \cite{BG}:

\begin{corollary}
The Hessian $D^2u$ has constant rank in $B$.
\end{corollary}

Indeed, applying Theorem \ref{mainthm} on appropriately scaled balls, the sets $$\{ x \in B \ | \ \textrm{rank}(D^2u(x)) \le k \}, \quad k=0,1,2, \ldots, n,$$ are open in $B$.  On the other hand the sets $\{ x \in B \ | \ \textrm{rank}(D^2u(x)) \ge k \}$ are open in $B$ by continuity of the eigenvalues  of $D^2u$.  A consequence is that the sets $$\{ x \in B \ | \ \textrm{rank}(D^2u(x)) =k \}$$ are open and closed in $B$, giving the corollary.

We now give an outline of the paper.  In Section \ref{sectionprelim} we recall some definitions and known results about semi-concave functions  and in particular we provide a proof of the semi-concavity of the sum of the first $k$ eigenvalues of $D^2u$ for $u$ in $C^4$.  We also give a version of the weak Harnack inequality for subsolutions of elliptic equations.  

In Section \ref{section3}, under the assumption that $u$ is in $C^4$, we prove that the key differential inequality
\begin{equation} \label{kdi}
F^{ab} Q_{ab} \le CQ + \sum_{i=1}^k b^{i,j} (\lambda_i)_j, \quad \textrm{for } Q=Q_k=\lambda_k + 2\lambda_{k-1} + \cdots + k \lambda_1,
\end{equation}
holds almost everywhere, where  $C$ and $b^{i,j}$ are bounded.  This improves on the analogous result in \cite{SW} where the inequality is proved for approximating polynomials.  We note that the method of Section \ref{section3} includes proofs of a first variation formula for $\lambda_i$ (Lemma \ref{lemmafir}) and a second variation inequality (Lemma \ref{lemmala}), which hold almost everywhere. 

In Section \ref{section4} we complete the proof of Theorem \ref{mainthm}.
We cannot directly apply the weak Harnack inequality to $Q$ satisfying (\ref{kdi}) because its right hand side includes derivatives of $\lambda_1, \ldots, \lambda_k$.  We get around this difficulty by considering a new quantity
$$R = \sum_{k=1}^{\ell} (Q_k +\ve)^{1/2}, \quad \textrm{for } \ve>0.$$
By exploiting the concavity of the square root function, $R$ is shown to satisfy the differential inequality
$$F^{ab} R_{ab} \le C R,$$
almost everywhere.  We apply the weak Harnack inequality to $R$ and then let $\ve \rightarrow 0$ to obtain Theorem \ref{mainthm}.

\section{Preliminaries} \label{sectionprelim}

In this section we collect some elementary and well-known results which we will need in the sequel.

\subsection{Semi-concave functions} Let $U$ be a bounded convex subset of $\mathbb{R}^n$.  
A real-valued function $f$ on $U$  is \emph{semi-concave} if there exists a constant $M$ such that $g=f-M |x|^2$ is concave.   We call $M$ the semi-concavity constant for $f$ on $U$.  Observe that every function in $C^2(\ov{U})$ is automatically semi-concave.

Equivalently, a continuous function $f$ is semi-concave if for some $M'$ 
\begin{equation}\label{sc}
\frac{f(x) + f(y)}{2} - f\left( \frac{x+y}{2} \right) \le M' |x-y|^2, \quad \textrm{for all } x,y \in U.
\end{equation}

It is a classical result that a concave function $f$ is Lipschitz continuous and hence differentiable almost everywhere (we write its derivative as $Df(x)$ if it exists at $x$).  Moreover, by a theorem of Alexandrov, the second derivative of $f$ exists almost everywhere in the sense that there is a second order Taylor expansion at almost every $x$ (see for example \cite[Theorem 2.6.4]{BV}).  This holds too then for semi-concave functions $f$ on $U$.  More explicitly, at almost every $x \in U$, the derivative $Df(x)=(f_1(x),\ldots, f_n(x))$ exists and there is a symmetric matrix which we write as $D^2f(x)=(f_{ij}(x))$ such that
\begin{equation} \label{taylor}
f(y) = f(x) + f_i(x) (y-x)_i + \frac{1}{2} f_{ij}(x) (y-x)_i (y-x)_j +o(|y-x|^2), \quad \textrm{as } |y-x| \rightarrow 0,
\end{equation}
where as usual we are summing repeated indices from $1$ to $n$.

The following proposition is well-known (see for example \cite{CS}), but we include a brief proof for the reader's convenience.

\begin{proposition} \label{propsemi}
Let $u \in C^4(\ov{U})$, for a bounded convex set $U \subset \mathbb{R}^n$.  Denote by $\lambda_1(x) \le \cdots \le \lambda_n(x)$ the eigenvalues of the Hessian $D^2u(x)$.  Then for each $k=1, \ldots, n$ the map $\ov{U} \rightarrow \mathbb{R}$ given by
$$x \mapsto \lambda_1(x) + \cdots + \lambda_k(x)$$
is semi-concave.
\end{proposition}
\begin{proof}
We claim the following:
the map $\sigma:\text{Sym}_n(\mathbb{R}) \rightarrow \mathbb{R}$ given by  $$\sigma(A)=  \lambda_1(A)+ \cdots + \lambda_k(A)$$ is increasing and concave on $\text{Sym}_n(\mathbb{R})$, and is Lipschitz continuous with Lipschitz constant depending only on $n$ and $k$.  To see the claim, note that
given fixed unit vectors $V_1, \ldots, V_k \in \mathbb{R}^n$, the function
 $$A \mapsto \sum_{\alpha=1}^k A_{ij} V_{\alpha}^i V_{\alpha}^j,$$
 is linear, increasing and has bounded Lipschitz constant depending only on $n$ and $k$.  Here we are writing $V_{\alpha}= (V_{\alpha}^1, \ldots, V_{\alpha}^n)$ and $A=(A_{ij})_{i,j=1}^n$.
 But we can define
 $$\sigma(A) = \inf \left\{ \sum_{\alpha=1}^k A_{ij} V_{\alpha}^i V_{\alpha}^j \ | \ V_1, \ldots, V_k \textrm{ are orthonormal} \right\}.$$
The map $\sigma$ is clearly increasing, and it is concave since the infimum of concave functions is
  concave.  Moreover, it is an elementary fact that for any normed vector space $(X, \| \cdot \|)$, if $f_{s}: X \rightarrow \mathbb{R}$, for $s\in S$, is a family of functions which are uniformly Lipschitz continuous:
$$|f_{s}(x) - f_{s}(y)| \le C\|x-y\|, \quad x,y \in X$$
 then $f := \min_{s \in S} f_{s}$, assuming the minimum is attained at each point and is finite, is also Lipschitz continuous with the same constant $C$.  The claim follows.

For $x, y \in \ov{U}$, using concavity of $\sigma$, 
\[
\begin{split}
\lefteqn{\frac{\sigma(D^2u(x))+\sigma(D^2u(y))}{2} - \sigma \left( D^2u \left( \frac{x+y}{2} \right)  \right) } \\ \le {} & \sigma \left( \frac{D^2u(x) + D^2u(y)}{2} \right) - \sigma \left( D^2u \left( \frac{x+y}{2} \right)  \right).
\end{split}
\]
But then since $u$ is in $C^4(\ov{U})$ we have that $D_{ij} u$ is semi-concave for every $i,j$ and hence 
$$\frac{D_{ij}u(x) + D_{ij}u(y)}{2} - D_{ij}u\left( \frac{x+y}{2} \right) \le  M|x-y|^2.$$
Then, increasing $M$ if necessary, we have the following inequality of symmetric matrices:
$$\frac{D^2u(x) + D^2u(y)}{2} - D^2u\left( \frac{x+y}{2} \right) \le  M|x-y|^2 \textrm{Id}.$$
Since $\sigma$ is increasing and Lipschitz continuous,
\[
\begin{split}
\lefteqn{\frac{\sigma(D^2u(x))+\sigma(D^2u(y))}{2} - \sigma \left( D^2u \left( \frac{x+y}{2} \right)  \right) } \\ \le {} & \sigma \left( D^2u\left( \frac{x+y}{2} \right) + M|x-y|^2 \textrm{Id} \right) - \sigma \left( D^2u \left( \frac{x+y}{2} \right)  \right) \\
\le {} & C |x-y|^2,
\end{split}
\]
completing the proof of the proposition.
\end{proof}

We end the subsection with an elementary proposition.

\begin{proposition} \label{comp} Let $U$ be a bounded convex  set in $\mathbb{R}^n$ and $V$ an open interval in $\mathbb{R}$.  Suppose that  $f: U \rightarrow V$ is semiconcave and $h: V\rightarrow \mathbb{R}$ is increasing, Lipschitz continuous and concave.  Then $h \circ f: U \rightarrow \mathbb{R}$ is semiconcave.
\end{proposition}
\begin{proof}
We have for $x,y \in U$,
\[
\begin{split}
\lefteqn{\frac{h(f(x)) + h(f(y))}{2} - h \left( f \left(\frac{x+y}{2} \right) \right) } \\ \le {} & h \left( \frac{f(x)+f(y)}{2} \right) - h \left( f \left(\frac{x+y}{2} \right) \right) \qquad \textrm{($h$ concave)}\\
\le {} & h\left( f\left( \frac{x+y}{2} \right) + M|x-y|^2 \right) - h \left( f \left(\frac{x+y}{2} \right) \right)  \qquad \textrm{($f$ semiconcave, $h$ increasing)} \\
\le {} & C |x-y|^2, \qquad \textrm{($h$ Lipschitz)}
\end{split}
\]
as required.
\end{proof}

\subsection{The weak Harnack inequality}

We state a weak Harnack inequality for semi-concave functions.  It follows by approximation from the classical weak Harnack inequality for functions in $W^{2,n}$ (see \cite[Theorem 9.22]{GT}).  Similar statements, with slightly different hypotheses, can be found in \cite{Tr} and \cite{SW}.

\begin{proposition} \label{lemmaweak}
Consider the operator $Lv = a^{ij} D_{ij} v + b^i D_i v + cv$ with bounded coefficients on the unit ball $B \subset \mathbb{R}^n$ and with $a^{ij}$ satisfying the uniform ellipticity condition $\Lambda^{-1}|\xi|^2 \le a^{ij} \xi_i \xi_j \le \Lambda |\xi|^2$ for all $\xi \in \mathbb{R}^n$ for   $\Lambda>0$.  Let $v$ be a semi-concave nonnegative function on  $B$ satisfying $Lv \le f$ almost everywhere in $B$ for $f \in L^n(B)$.  Then on the half size ball $B'$,
\begin{equation} \label{v}
\left( \frac{1}{|B'|} \int_{B'} v^q \right)^{1/q} \le C  \left( \inf_{B'} v + \| f \|_{L^n(B)} \right),
\end{equation}
for positive constants $C$ and $q$ depending only on $n, \Lambda,$ bounds for $b^i$ and $c$ and the radius of the ball $B$.
\end{proposition}
\begin{proof}  We include a brief proof for the sake of completeness.  Let $\tilde{B}$ be a ball such that $B' \subset \subset \tilde{B} \subset \subset B$.
Since $v$ is semi-concave, it is Lipschitz continuous and twice differentiable almost everywhere.  For $\ve>0$ let $v_{\ve}$ be a standard mollification of $v$.  Then $v_{\ve} \rightarrow v$ uniformly and $D^kv_{\ve} \rightarrow D^kv$ for $k=1,2$ almost everywhere on $\tilde{B}$ as $\ve \rightarrow 0$ (the second assertion is a direct consequence of the expansion (\ref{taylor})).  Let $\delta>0$ be given.  Then by Egorov's Theorem there exists a set $K_{\delta} \subset \tilde{B}$ such that $|\tilde{B}\setminus K_{\delta}| \le \delta$ and 
 $D^k v_{\ve} \rightarrow D^kv$  uniformly on $K_{\delta}$ for $k=1,2$.  Then we may choose $\ve>0$ sufficiently small so that 
 $$L v_{\ve} = Lv + L(v_{\ve}-v) \le f+ \delta, \qquad \textrm{almost everywhere on $K_{\delta}$}.$$
 On the other hand, since $v$ is semi-concave we have an upper bound for $D^2 v_{\ve}$ and hence on $\tilde{B} \setminus K_{\delta}$ 
 $$L v_{\ve} \le M,$$
 for a constant $M \ge 1$ depending on the $C^0$ norm, semi-concavity constant and Lipschitz bound of $v$ as well as bounds on the coefficients of $L$.
 It follows that almost everywhere on $\tilde{B}$ we have
 $$L v_{\ve}\le f +g,$$
 for a function $g \in L^{\infty}(\tilde{B})$ with $\| g \|_{L^n(\tilde{B})} \le CM \delta^{1/n}$.   Then we apply \cite[Theorem 9.22]{GT}  to the smooth nonnegative function $v_{\ve}$ to obtain
 $$\left( \int_{B'} v_{\ve}^q \right)^{1/q} \le C \inf_{B'} v_{\ve} + C \| f +g\|_{L^n(\tilde{B})} \le C\left(  \inf_{B'} v_{\ve} + \|f \|_{L^n(B)} +M \delta^{1/n} \right),$$
 for uniform positive constants $C, q$.  Then let $\delta\rightarrow 0$, so that in addition $\ve \rightarrow 0$ and we obtain (\ref{v}).
\end{proof}

We emphasize that the constants $C, q$ in the above proposition are independent of the semi-concavity constant of $v$.

\section{A differential inequality} \label{section3}

As in the introduction, let $u$ solve the equation \eqref{maineqn} subject to the conditions 
 \eqref{convexcondition} and \eqref{ellipticitycon}.  In this section, we make the following: 

\medskip
\noindent
{\bf Additional assumption:} $u \in C^4(B).$

\medskip

 Let $0\le \lambda_1 \le \cdots \le \lambda_n$ be the eigenvalues of $D^2u$.
By Proposition \ref{propsemi}, the map $x \mapsto \lambda_1 + \cdots + \lambda_k$ is semi-concave on compact convex subsets of $B$.  It follows that $\lambda_1, \ldots, \lambda_n$ are twice differentiable almost everywhere on $B$.  
The goal of this section is the following differential inequality.

\begin{lemma} \label{lemmadi} For $1 \le k \le n$,  define 
$$Q := Q_k:= \lambda_k + 2 \lambda_{k-1} + \cdots + k \lambda_1.$$
Then
\begin{equation} \label{claimotro}
F^{ab} Q_{ab} \le C Q + \sum_{i=1}^k b^{i,j} (\lambda_i)_j, \quad \textrm{almost everywhere on $B$,}
\end{equation}
where $C$ is a uniform constant and the $b^{i,j}$ are uniformly bounded functions on $B$.
\end{lemma}

In the above, ``uniform'' means that the constants depend only on $n$, $\Lambda$, $\| u \|_{C^3(B)}$ and $\| F \|_{C^2}$.  In particular, the constants do not depend on a $C^4$ bound for $u$. 

To establish the lemma, we prove two rather general results about the first and second derivatives of the eigenvalues $\lambda_i$, which will hold almost everywhere.  In the case when the eigenvalues $\lambda_i$ are all distinct, there are well-known formulae for their first and second derivatives (see for example \cite{Sp}).  To deal with eigenvalues with multiplicity we adapt an approach  of Brendle-Choi-Daskalopoulos \cite[Lemma 5]{BCD} where similar statements to Lemmas \ref{lemmafir} and \ref{lemmala} below are proved for $\lambda_1$.   Crucially, these lemmas only hold at a point where the eigenvalues are twice differentiable.

Fix an $x_0$ at which the $\lambda_i$ are twice differentiable, and choose coordinates at $x_0$ such that $D^2u$ is diagonal with entries $u_{ii} = \lambda_i$.    Moreover, we suppose that $N$ of the $\lambda_i$'s are distinct at $x_0$, and define $1 \le \mu_1 < \ldots < \mu_N=n$ by
$$\lambda_1  = \cdots = \lambda_{\mu_1} < \lambda_{1+\mu_1} = \cdots = \lambda_{\mu_2} < \lambda_{1+\mu_2} = \cdots \cdots = \lambda_{\mu_N}=\lambda_n.$$
We also define $\mu_0=0$ so that the multiplicities of the eigenvalues of $D^2u$ at $x_0$ are $\mu_1-\mu_0, \mu_2-\mu_1, \ldots, \mu_N-\mu_{N-1}$.

\begin{lemma} \label{lemmafir} For each  $j=1, 2, \ldots, N$ we have at $x_0$,
\begin{equation} \label{fir}
u_{k\ell i} = (\lambda_{1+\mu_{j-1}})_i \delta_{k\ell}, \qquad \textrm{for } 1+\mu_{j-1} \le k, \ell \le \mu_{j}.
\end{equation}
for $i=1, 2, \ldots, n$.
\end{lemma}
\begin{proof}
We prove this by induction on $j$.  We first prove the case $j=1$, which states that
\begin{equation} \label{first}
u_{k\ell i} = (\lambda_1)_i \delta_{k\ell}, \quad \textrm{for } 1 \le k, \ell \le \mu_1, \quad 1 \le i \le n.
\end{equation}

Let $V = (V^1, \ldots, V^n)$ be a constant unit vector field defined in a neighborhood of $x_0$.  Then by definition of $\lambda_1$, the function $h$ defined by
$$h:= u_{k\ell} V^k V^{\ell} - \lambda_1,$$
has $h(x) \ge 0$ for $x$ near $x_0$.  Choose $V$ with $V^k(x_0)=0$ for $k > \mu_1$ so that we have $h(x_0)=0$ and $h$ has a local minimum at $x_0$.  Moreover, $h$ is twice differentiable at $x_0$.

Then at $x_0$,
$$0=h_i = \sum_{k, \ell \le \mu_1} u_{k\ell i} V^k V^{\ell} - \sum_{k, \ell \le \mu_1} (\lambda_1)_i \delta_{k\ell} V^k V^{\ell},$$
using the fact that $V$ is a unit vector.  Then (\ref{first}) follows since we are free to choose $V^k(x_0)$ for $k \le \mu_1$.

For the inductive step, assume (\ref{fir}) holds for $1 \le j \le p$.   Let $V_1, \ldots, V_{\mu_{p}}$ be the constant unit vector fields in the $\partial/\partial x_1, \ldots, \partial/\partial x_{\mu_p}$ directions.  That is, writing $V_{\alpha}$ in component form as $V_{\alpha} = (V^1_{\alpha}, \ldots, V^n_{\alpha})$, we have $V_{\alpha}^q = \delta_{q\alpha}$.  Next let $W=V_{1+\mu_p}$ be an arbitrary constant unit vector field in the span of
the directions $\partial/\partial x_{1+\mu_p}, \ldots, \partial/\partial x_{\mu_{p+1}}$. 

 Consider the quantity
\[
\begin{split}
h= {} & \sum_{\alpha=1}^{1+\mu_p} u_{k\ell} V_{\alpha}^k V_{\alpha}^{\ell} - \sum_{j=1}^{p} (\mu_j-\mu_{j-1}) \lambda_{1+\mu_{j-1}} - \lambda_{1+\mu_p}\\
= {} &  \sum_{\alpha=1}^{1+\mu_p} u_{k\ell} V_{\alpha}^k V_{\alpha}^{\ell} - \sum_{j=1}^{p} \sum_{1+\mu_{j-1} \le \alpha \le \mu_{j}} \lambda_{1+\mu_{j-1}} \delta_{k\ell}V^k_{\alpha} V^{\ell}_{\alpha} - \lambda_{1+\mu_{p}} \\
= {} & \sum_{j=1}^{p} \sum_{1+\mu_{j-1} \le \alpha \le \mu_{j}} (u_{k\ell}-  \lambda_{1+\mu_{j-1}} \delta_{k\ell})V^k_{\alpha} V^{\ell}_{\alpha} + u_{k\ell} W^k W^{\ell} - \lambda_{1+\mu_{p}}.
\end{split}
\]
Now note that $h(x_0)=0$ and $h(x) \ge 0$ for $x$ near $x_0$.  Since $h$ achieves a minimum at $x_0$ we have
\[
\begin{split}
0 = h_i(x_0)={} & \sum_{j=1}^{p} \sum_{1+\mu_{j-1} \le \alpha \le \mu_{j}} (u_{k\ell i}-  (\lambda_{1+\mu_{j-1}})_i \delta_{k\ell})V^k_{\alpha} V^{\ell}_{\alpha}  + u_{k\ell i} W^k W^{\ell} - (\lambda_{1+\mu_{p}})_i \\
= {} & (u_{k\ell i} - \delta_{k\ell} (\lambda_{1+\mu_{p}})_i) W^k W^{\ell},
\end{split}
\]
where for the last line we used the inductive hypothesis and the fact that $W$ is a unit vector.
Since $W$ is an arbitrary unit vector in the span of $\partial/\partial x_{1+\mu_{p}}, \ldots, \partial/\partial x_{\mu_{p+1}}$, we have proved (\ref{fir}) for $j=p+1$ and the result follows by induction.
\end{proof}

As above, pick coordinates at $x_0$ such that $D^2u$ is diagonal with entries $u_{ii}=\lambda_i$.   Fix $m$ between $1$ and $n$.
Define $\rho \in \{ m, m+1, \ldots, n\}$ to be the largest integer such that $\lambda_{\rho} = \lambda_m$ at $x_0$, so that
$$0 \le \lambda_1 \le \cdots \le \lambda_m = \lambda_{m+1} = \cdots = \lambda_{\rho} < \lambda_{\rho+1} \le \cdots \le \lambda_n.$$
Then we have the following lemma on the second derivatives of the $\lambda_{\alpha}$.  We emphasize that this only holds at the point $x_0$ where the $\lambda_i$ are twice differentiable.

\begin{lemma}  \label{lemmala}
As symmetric $n\times n$ matrices we have at $x_0$,
\begin{equation} \label{secondg}
\sum_{\alpha=1}^m (\lambda_{\alpha})_{ab} \le \sum_{\alpha=1}^m u_{\alpha \alpha ab} + 2 \sum_{\alpha=1}^m \sum_{q>\rho} \frac{u_{q\alpha a} u_{q\alpha b}}{\lambda_{\alpha}-\lambda_q}.
\end{equation}
\end{lemma}
\begin{proof}
Let $V_1, \ldots, V_m$ be smooth mutually orthogonal unit vector fields defined in a neighborhood of $x_0$ with $V_{\alpha}(x_0)$ the unit vector in the $\partial/\partial x_{\alpha}$ direction.  In particular, writing $V_{\alpha} = (V^1_{\alpha}, \ldots, V^n_{\alpha})$ we have $V^q_{\alpha}=\delta_{q\alpha}$ at $x_0$.  

We consider the quantity
$$h(x) = \sum_{\alpha=1}^m u_{k\ell} V^k_{\alpha} V^{\ell}_{\alpha} - \sum_{\alpha=1}^m \lambda_{\alpha},$$
which has $h(x_0)=0$ and $h(x) \ge 0$ for $x$ near $x_0$.  In particular $h$ achieves a minimum at $x_0$, and moreover, $h$ is twice differentiable at $x_0$.

We prescribe the first and second derivatives of the $V_{\alpha}$ at $x_0$ as follows.  For $1 \le \alpha \le m$, and $1\le a \le n$,
$$\partial_a V^q_{\alpha} = \left\{ \begin{array}{ll} 0, \quad & q \le \rho \\ \frac{u_{\alpha qa}}{\lambda_{\alpha}-\lambda_q}, \quad & q > \rho. \end{array} \right.$$
For $1 \le \alpha, \beta \le m$ and $1\le a,b \le n$,
$$\partial_a \partial_b V^{\alpha}_{\beta} = - \frac{1}{2} \sum_{q>\rho} \frac{u_{\alpha qa} u_{\beta qb}}{(\lambda_{\alpha}-\lambda_q)(\lambda_{\beta}-\lambda_q)} - \frac{1}{2} \sum_{q>\rho} \frac{u_{\alpha qb} u_{\beta qa}}{(\lambda_{\alpha}-\lambda_q)(\lambda_{\beta}-\lambda_q)}$$
noting that when $\alpha=\beta$ we have
$$\partial_a \partial_b V^{\alpha}_{\alpha} = - \sum_{q>\rho} \frac{u_{\alpha qa} u_{\alpha qb}}{(\lambda_{\alpha}-\lambda_q)^2}.$$
We take $\partial_a \partial_b V^q_{\alpha}=0$ with $q >m$.

We first check these definitions are consistent with the $V_{\alpha}$ being orthonormal vectors.  Compute at $x_0$, for $\alpha, \beta=1, \ldots, m$,
$$\partial_a \left( \sum_q V^q_{\alpha} V^q_{\beta} \right) = \sum_q (\partial_a V^q_{\alpha}) V^q_{\beta} + \sum_q V^q_{\alpha} \partial_a V^q_{\beta}=0,$$
since $\partial_a V^q_{\alpha}$ and $\partial_a V^q_{\beta}$ vanish when $q \le m \le \rho$.  And
\[
\begin{split}
\partial_a \partial_b \left( \sum_q V_{\alpha}^q V_{\beta}^q \right) = {} & \sum_{q >\rho} (\partial_a V^q_{\alpha})(\partial_b V^q_{\beta}) +  \sum_{q >\rho} (\partial_b V^q_{\alpha})(\partial_a V^q_{\beta})   +  \partial_a \partial_b V^{\beta}_{\alpha} +  \partial_a \partial_b V^{\alpha}_{\beta} \\
= {} & \sum_{q > \rho} \frac{u_{\alpha qa} u_{\beta qb}}{(\lambda_{\alpha}-\lambda_q)(\lambda_{\beta}-\lambda_q)} +\sum_{q > \rho} \frac{u_{\alpha qb} u_{\beta qa}}{(\lambda_{\alpha}-\lambda_q)(\lambda_{\beta}-\lambda_q)} \\
{} &-  \sum_{q>\rho} \frac{u_{\alpha qa} u_{\beta qb}}{(\lambda_{\alpha}-\lambda_q)(\lambda_{\beta}-\lambda_q)} -  \sum_{q>\rho} \frac{u_{\alpha qb} u_{\beta qa}}{(\lambda_{\alpha}-\lambda_q)(\lambda_{\beta}-\lambda_q)} \\
= {} & 0,
\end{split}
\]
as required.

Since $h$ has a minimum at $x_0$ we have the inequality of matrices:
\[
\begin{split}
0 \le h_{ab} = {} & \sum_{\alpha=1}^m \left\{  u_{\alpha \alpha ab} - (\lambda_{\alpha})_{ab} +  2u_{k\ell a} (\partial_b V_{\alpha}^k) V_{\alpha}^{\ell} + 2u_{k\ell b} (\partial_a V_{\alpha}^k) V_{\alpha}^{\ell} \right. \\
{} & \left. + 2 u_{k\ell} (\partial_a V^k_{\alpha}) (\partial_b V^{\ell}_{\alpha}) + 2u_{k\ell} (\partial_a \partial_b V^k_{\alpha}) V^{\ell}_{\alpha} \right\} \\
= {}  & \sum_{\alpha=1}^m \left\{  u_{\alpha \alpha ab} - (\lambda_{\alpha})_{ab} + 4 \sum_{q >\rho} \frac{u_{q\alpha a} u_{\alpha q b}}{\lambda_{\alpha}-\lambda_q} \right. \\
{} &\left. + 2\sum_{q > \rho} \lambda_q \frac{u_{\alpha qa} u_{\alpha qb}}{(\lambda_{\alpha}-\lambda_q)^2} - 2 \lambda_{\alpha} \sum_{q>\rho} \frac{u_{\alpha qa} u_{\alpha qb}}{(\lambda_{\alpha}-\lambda_q)^2} \right\} \\ 
= {} &  \sum_{\alpha=1}^m \left\{  u_{\alpha \alpha ab} - (\lambda_{\alpha})_{ab} + 2 \sum_{q >\rho} \frac{u_{q\alpha a} u_{\alpha q b}}{\lambda_{\alpha}-\lambda_q} \right\},
\end{split}
\]
giving (\ref{secondg}).
\end{proof}

We can now complete the proof of Lemma \ref{lemmadi},   making crucial use of the convexity condition.  Following \cite{BG}, 
we observe that the convexity condition (\ref{convexcondition}) can be written as:  for every symmetric matrix $(X_{ab}) \in \textrm{Sym}_n(\mathbb{R})$, vector $(Z_a) \in \mathbb{R}^n$ and $Y \in \mathbb{R}$, 
\begin{equation} \label{keyco}
\begin{split}
0 \le {} &  F^{ab,rs} X_{ab} X_{rs} + 2   F^{ar} A^{bs} X_{ab} X_{rs} +  F^{x_a, x_b} Z_a Z_b \\
& {} - 2 F^{ab,u} X_{ab} Y - 2  F^{ab, x_r} X_{ab} Z_r + 2 F^{u,x_a} Y Z_a + F^{u,u} Y^2, 
\end{split}
\end{equation}
where we are evaluating the derivatives of $F$ at $(A, p, u, x)$ for a positive definite matrix $A$.  We are writing $A^{ij}$ for the  $(i,j)$th entry of the inverse matrix of $A$.  We remark that we do not require the condition (\ref{convexcondition}) to hold for the entire space $\textrm{Sym}_n^+(\mathbb{R}) \times \mathbb{R} \times B$, but only for a certain subset which depends on the solution $u$.

\begin{proof}[Proof of Lemma \ref{lemmadi}]
Observe that
$$Q = \sum_{m=1}^k \sum_{\alpha =1}^m \lambda_{\alpha}.$$
We will compute at a point $x_0$ at which the $\lambda_i$ are twice differentiable and $D^2u$ is diagonal with entries $u_{ii} =\lambda_i$.
From Lemma \ref{lemmala} we have
\[
\begin{split}
F^{ab} Q_{ab} = {} &  \sum_{m=1}^k \sum_{\alpha=1}^m F^{ab} (\lambda_{\alpha})_{ab} \\
\le {} & \sum_{m=1}^k \sum_{\alpha=1}^m F^{ab} u_{\alpha \alpha ab} + 2\sum_{m=1}^k \sum_{\alpha=1}^m \sum_{q>\rho_m}F^{ab} \frac{u_{q\alpha a}u_{q\alpha b}}{\lambda_{\alpha} - \lambda_q},
\end{split}
\]
where we are writing $\rho_m$ for the largest integer $\rho_m \in \{m, m+1, \ldots,   n\}$ with the property that $\lambda_{\rho_m} = \lambda_m$ at $x_0$.  Differentiating the equation
$$F(D^2u, Du, u,x)=0,$$
twice in the $x_{\alpha}$ direction gives
\begin{equation} \label{dt}
\begin{split}
0 = {} & F^{ab} u_{ab\al \alpha} + F^{p_a} u_{a\al \al} + F^u u_{\al \al}  \\
{} & + F^{ab, rs} u_{ab\al} u_{rs\al} + F^{p_a, p_b} u_{a\al} u_{b\al} + F^{u,u} u_{\al}^2 + F^{x_{\al}, x_{\al}} \\
{} & + 2F^{ab, p_r} u_{ab\al} u_{r\al} + 2F^{ab,u} u_{ab\al} u_{\al} + 2F^{ab,x_{\al}} u_{ab{\al}} \\
{}& + 2F^{p_a, u} u_{a\al} u_{\al} + 2F^{p_a, x_{\al}} u_{a{\al}} + 2F^{u, x_{\al}} u_{\al}. 
\end{split}
\end{equation}
Hence,
\[
\begin{split}
F^{ab} Q_{ab} \le {} & \sum_{m=1}^k  \sum_{\alpha=1}^m \bigg\{2 \sum_{q>\rho_m}F^{ab} \frac{u_{q\alpha a}u_{q\alpha b}}{\lambda_{\alpha} - \lambda_q} - \sum_{a,b,r,s > \rho_k} F^{ab,rs} u_{ab\al} u_{rs\al} - F^{u,u} u_{\al}^2 - F^{x_{\al}, x_{\al}} \\
{} & - 2 \sum_{a,b>\rho_k} F^{ab, u} u_{ab\al}u_{\al} - 2 \sum_{a,b>\rho_k} F^{ab, x_{\al}} u_{ab\al} - 2F^{u, x_{\al}} u_{\al}  \bigg\} \\
{} & + C Q + \sum_{i=1}^k b^{i,j} (\lambda_i)_j +  \sum_{1 \le \alpha< \beta \le \rho_k} c^{j, \alpha, \beta} u_{\alpha \beta j},
\end{split}
\]
for uniformly bounded $b^{i,j}$, $c^{j, \alpha, \beta}$.  Here we are using Lemma \ref{lemmafir} which implies that if $1 \le \alpha \le \rho_k$ then at $x_0$,
$$u_{\alpha \alpha j} = (\lambda_i)_j,$$
for some $1\le i \le k$ with $\lambda_i=\lambda_{\alpha}$.
  But
\begin{equation}
\begin{split}
 \sum_{1 \le \alpha< \beta \le \rho_k} c^{j, \alpha, \beta} u_{\alpha \beta j} \le {} & C \sum_{1 \le \alpha <\beta \le \rho_k} (\lambda_{\beta}-\lambda_{\alpha}) +  2\sum_{1 \le \alpha <\beta \le \rho_k, \ \lambda_{\alpha}\neq \lambda_{\beta}} F^{ab} \frac{u_{\alpha \beta a} u_{\alpha \beta b}}{\lambda_{\beta}-\lambda_{\alpha}}   \\
\le {} & CQ + 2\sum_{\alpha=1}^{k-1} \sum_{\rho_{\alpha} < q \le \rho_k} \frac{F^{ab} u_{\alpha qa} u_{\alpha qb}}{\lambda_q -\lambda_{\alpha}} \\
\le {} &  CQ +  2\sum_{m=1}^{k-1} \sum_{\alpha=1}^m \sum_{\rho_m < q \le \rho_k  }F^{ab} \frac{u_{q\alpha a}u_{q\alpha b}}{\lambda_{q} - \lambda_{\al}},\end{split}
\end{equation}
where for the first line we note that if  $1 \le \alpha <\beta \le \rho_k$ with $\lambda_{\alpha}=\lambda_{\beta}$  then by Lemma \ref{lemmafir} we have $u_{\alpha \beta j}=0$.

We now apply the convexity assumption (\ref{keyco}), taking for each fixed $\alpha$,
\begin{equation} \label{choice}
X_{ab} = \left\{ \begin{array}{ll} - u_{ab\alpha} & \text{ if }a,b > \rho_k \\
0  & \text{ otherwise}, \end{array} \right.
\quad  
Z_a = \left\{ \begin{array}{ll} 1 & \text{ if } a=\alpha \\ 0 & \text{ otherwise,} \end{array} \right.
\quad Y = u_\alpha.
\end{equation}
This gives for each $\alpha=1, \ldots, m$,
\begin{equation}
\begin{split}
0 \le {} & \sum_{a,b,r,s >\rho_k} F^{ab, rs} u_{ab\al} u_{rs\al} + 2 \sum_{a,b,q>\rho_k} F^{ab} \frac{u_{qa\al} u_{qb\al}}{\lambda_q} + F^{x_{\al}, x_{\al}} \\
{} & + 2 \sum_{a,b>\rho_k} F^{ab, u} u_{ab{\al}} u_{\al} + 2\sum_{a,b>\rho_k} F^{ab, x_{\al}} u_{ab\al} +2F^{u, x_{\al}} u_{\al} + F^{u,u} u_{\al}^2.
\end{split}
\end{equation}

Observe that
\begin{equation}
2\sum_{m=1}^k \sum_{\alpha=1}^m \sum_{a,b,q>\rho_k} F^{ab} \frac{u_{qa\al} u_{qb\al}}{\lambda_q} \le 2 \sum_{m=1}^k \sum_{\al=1}^m \sum_{q > \rho_k} F^{ab} \frac{u_{q\alpha a} u_{q\alpha b}}{\lambda_q-\lambda_{\al}}.
\end{equation}
Combining all of the above gives
\begin{equation}
F^{ab}Q_{ab} \le CQ + \sum_{i=1}^k b^{i,j} ( \lambda_i)_j,
\end{equation}
as required.  \end{proof}

\section{Proof of Theorem \ref{mainthm}} \label{section4}

We will first assume that $u \in C^4(B)$ and then remove this assumption by approximation.
The quantity $Q_k = \lambda_k + 2\lambda_{k-1} + \cdots + k\lambda_1$ is semi-concave and from Lemma \ref{lemmadi} we have almost everywhere
$$F^{ab} (Q_k)_{ab} \le C Q_k +   \sum_{i=1}^k b^{i,j} (\lambda_i)_j,$$
for $b^{i,j}$ bounded.  For $\ve>0$ and a fixed $\ell \in \{1, 2, \ldots, n\}$, consider 
$$R = \sum_{k=1}^{\ell} (Q_k+\ve)^{1/2}.$$
Then by Proposition \ref{comp}, since the map $x \mapsto (x+\ve)^{1/2}$ is increasing, Lipschitz continuous and concave for $x \ge 0$, we see that $R$ is semiconcave (but note that its constant of semi-concavity will depend on $\ve$).  $R$ is twice differentiable almost everywhere.  At such a point, we compute
\[
\begin{split}
F^{ab} R_{ab} = {} & \frac{1}{2} \sum_{k=1}^{\ell} (Q_k+\ve)^{-1/2} F^{ab} (Q_k)_{ab} - \frac{1}{4} \sum_{k=1}^{\ell} (Q_k+\ve)^{-3/2} F^{ab}(Q_k)_a (Q_k)_b \\ 
\le {} & \frac{1}{2} \sum_{k=1}^{\ell} (Q_k+\ve)^{-1/2} (CQ_k + \sum_{i=1}^k b^{i,j} ( \lambda_i)_j) - c\sum_{k=1}^{\ell} (Q_k+\ve)^{-3/2} |DQ_k|^2 \\
\le {} & CR + C \sum_{k=1}^{\ell} (Q_k+\ve)^{-1/2} \sum_{i=1}^k |D\lambda_i| -  c\sum_{k=1}^{\ell} (Q_k+\ve)^{-3/2} |DQ_k|^2,
\end{split}
\]
for uniform constants $C, c>0$ (the constant $C$ may differ from line to line).  We claim that
\begin{equation} \label{Qkclaim}
\sum_{k=1}^{\ell} (Q_k+\ve)^{-1/2} \sum_{i=1}^k |D\lambda_i| \le C \sum_{k=1}^{\ell} (Q_k+\ve)^{-1/2} |DQ_k|,
\end{equation}
for a universal constant $C$.
We do this by induction on $\ell$.  The case $\ell=1$ is trivial.  Assume it holds for $\ell-1$.
We need to show that
\[
\begin{split}
(Q_{\ell}+\ve)^{-1/2}  \sum_{i=1}^{\ell} |D\lambda_i|  \le C\sum_{k=1}^{\ell} (Q_k+\ve)^{-1/2} |DQ_k|.
\end{split}
\]
But using $Q_{\ell-1} \le Q_{\ell}$, the inductive hypothesis and the definition of $Q_{\ell}$ we have
\[
\begin{split}
(Q_{\ell}+\ve)^{-1/2}  \sum_{i=1}^{\ell} |D\lambda_i|  \le {} &  (Q_{\ell-1}+\ve)^{-1/2} \sum_{i=1}^{\ell-1} |D \lambda_i| + (Q_{\ell}+\ve)^{-1/2} |D\lambda_{\ell}| \\
\le {} & C\sum_{k=1}^{\ell-1} (Q_k+\ve)^{-1/2} |DQ_k| + (Q_{\ell}+\ve)^{-1/2} \left(|D Q_{\ell}| + C \sum_{i=1}^{\ell-1} |D \lambda_i| \right) \\
\le {} & C\sum_{k=1}^{\ell-1} (Q_k+\ve)^{-1/2} |DQ_k| + (Q_{\ell}+\ve)^{-1/2} |D Q_{\ell}| \\ {} & + (Q_{\ell-1}  +\ve)^{-1/2} C \sum_{i=1}^{\ell-1} |D \lambda_i | \\
\le {} & C\sum_{k=1}^{\ell} (Q_k+\ve)^{-1/2} |DQ_k|.
\end{split}
\]
This completes the proof of (\ref{Qkclaim}).

It follows that, for constants $C, c>0$ independent of $\ve$,
\[
\begin{split}
F^{ab} R_{ab} \le  {} & C R + C \sum_{k=1}^{\ell} (Q_k+\ve)^{-1/2} |D Q_k| - c \sum_{k=1}^{\ell} (Q_k+\ve)^{-3/2} |DQ_k|^2 \\
\le {} &  C R,
\end{split}
\]
using the bound
$$(Q_k+\ve)^{-1/2} |D Q_k| \le \frac{\delta}{2} (Q_k+\ve)^{-3/2} |DQ_k|^2 + \frac{1}{2\delta} (Q_k+\ve)^{1/2},$$
which holds for any $\delta>0$.

Since $R$ is semi-concave, the weak Harnack inequality (Proposition \ref{lemmaweak} above) implies that for a uniform $q>0$ and $C$,
\begin{equation} 
\| R \|_{L^q(B_{1/2})} \le C \inf_{B_{1/2}} R.
\end{equation}
In particular, the constant $C$ is independent of $\ve$.  Hence we can let $\ve \rightarrow 0$ to obtain the same estimate for $\sum_{k=1}^{\ell} Q_k^{1/2}$.   Write 
$$S = \left( \sum_{k=1}^{\ell} Q_k^{1/2} \right)^2.$$

Now observe that for a  $C$ depending only on $n$ we have
$$\frac{1}{C} \lambda_{\ell} \le S \le C  \lambda_{\ell}.$$
We have
$$\| S \|_{L^{q/2}(B_{1/2})}^{1/2} \le C \inf_{B_{1/2}} S^{1/2},$$
and squaring both sides gives
$$\| S \|_{L^{q/2}(B_{1/2})} \le C \inf_{B_{1/2}} S$$
and hence
$$\|  \lambda_{\ell} \|_{L^{q/2}(B_{1/2})} \le C \inf_{B_{1/2}}  \lambda_{\ell}.$$
This completes the proof of the theorem in the case that $u \in C^4(B)$.  

For $u \in C^3(B)$ as in the statement of the theorem, the elliptic equation satisfied by $u$ implies that $u \in W_{\textrm{loc}}^{4,p}(B)$ for all $p$.  Fix $p>n$ and a ball $\tilde{B}$ with $B_{1/2} \subset \subset \tilde{B}\subset \subset B$.  Then we can find a sequence of smooth convex functions $u^{(s)}$ in a neighborhood of $\tilde{B}$ such that $u^{(s)} \rightarrow u$ in  $W^{4,p}(\tilde{B})$ as $s\rightarrow \infty$.  This also implies that $u^{(s)} \rightarrow u$ in $C^3(\tilde{B})$.  We wish to apply Lemma \ref{lemmadi} to each $u^{(s)}$.  Although $u^{(s)}$ does not solve $F(D^2u, Du, u, x)=0$ we see from (\ref{dt}) and  $F\in C^2$ that $\tilde{u} := u^{(s)}$ satisfies
\begin{equation} \label{dt2}
\begin{split}
0 = {} & \tilde{F}^{ab} \tilde{u}_{ab\al \alpha} + \tilde{F}^{p_a} \tilde{u}_{a\al \al} + \tilde{F}^u \tilde{u}_{\al \al}  \\
{} & + \tilde{F}^{ab, rs} \tilde{u}_{ab\al} \tilde{u}_{rs\al} + \tilde{F}^{p_a, p_b} \tilde{u}_{a\al} \tilde{u}_{b\al} + \tilde{F}^{u,u} \tilde{u}_{\al}^2 + \tilde{F}^{x_{\al}, x_{\al}} \\
{} & + 2\tilde{F}^{ab, p_r} \tilde{u}_{ab\al} \tilde{u}_{r\al} + 2\tilde{F}^{ab,u} \tilde{u}_{ab\al} \tilde{u}_{\al} + 2\tilde{F}^{ab,x_{\al}} \tilde{u}_{ab{\al}} \\
{}& + 2\tilde{F}^{p_a, u} \tilde{u}_{a\al} \tilde{u}_{\al} + 2\tilde{F}^{p_a, x_{\al}} \tilde{u}_{a{\al}} + 2\tilde{F}^{u, x_{\al}} \tilde{u}_{\al} - f^{(s)}, 
\end{split}
\end{equation}
where $f^{(s)} \rightarrow 0$ in $L^p(\tilde{B})$ as $s \rightarrow \infty$.   Here we are using $\tilde{F}$ to indicate that we are evaluating the derivatives of $F$  at $\tilde{u}$.  Carrying out the rest of the proof of Lemma \ref{lemmadi} with this extra term $f^{(s)}$ gives almost everywhere on $\tilde{B}$,
$$\tilde{F}^{ab} \tilde{Q}_{ab} \le C \tilde{Q} + \sum_{i=1}^k b^{i,j}( \tilde{\lambda}_i)_j + f^{(s)},$$
modifying $f^{(s)}$ if necessary, and  evaluating $\tilde{\lambda}_i$, $\tilde{Q}$ etc with respect to $\tilde{u}$.  Then $\tilde{R}$ satisfies almost everywhere on $\tilde{B}$,
$$\tilde{F}^{ab} \tilde{R}_{ab} \le C \tilde{R} + \frac{1}{2}\ve^{-1/2} f^{(s)}.$$
Applying Proposition \ref{lemmaweak} we obtain
$$\| \tilde{R} \|_{L^q(B_{1/2})} \le C \inf_{B_{1/2}} \tilde{R} + C \ve^{-1/2} \| f^{(s)} \|_{L^n(\tilde{B})}$$
and letting $s \rightarrow \infty$ gives
$$\| R \|_{L^q(B_{1/2})} \le C \inf_{B_{1/2}} R,$$
for a constant $C$ independent of $\ve$.
The rest of the proof goes through as above, and this completes the proof of Theorem \ref{mainthm}.

\end{document}